\documentclass[preprint]{elsarticle}
\usepackage{aliascnt,amsmath,amssymb,amsthm,colonequals,enumitem,mathtools,mleftright,stmaryrd}
\usepackage[unicode]{hyperref}

\newcommand{\newaliastheorem}[2]{%
  \newaliascnt{#1}{theorem}
  \newtheorem{#1}[#1]{#2}
  \aliascntresetthe{#1}
  \expandafter\def\csname #1autorefname\endcsname{#2}
}

\theoremstyle{plain}
\newtheorem{theorem}{Theorem}
\newaliastheorem{lemma}{Lemma}
\newaliastheorem{proposition}{Proposition}
\newaliastheorem{corollary}{Corollary}

\newcommand{\tuple}[1]{\left\langle #1 \right\rangle}

\newcommand{\Ord}{\mathrm{Ord}}

\DeclareMathOperator{\dd}{def}
\DeclareMathOperator{\trcl}{trcl}

\DeclareMathOperator{\din}{\dot{\in}}
\DeclareMathOperator{\dnin}{\dot{\not\in}}

\DeclareMathOperator{\deq}{\dot{=}}
\DeclareMathOperator{\dseq}{\dot{\subseteq}}

\begin{document}

\begin{frontmatter}
  \title{An equiconsistency proof for \texorpdfstring{$\mathrm{CZF} + V = L$}{CZF + V = L}}
  \author{Shuwei Wang}
  \ead{mmsw@leeds.ac.uk}
  \affiliation{organization={School of Mathematics, University of Leeds},city={Leeds},citysep={},postcode={LS2 9JT},country={UK}}

  \begin{abstract}
    In many axiomatic set theories, G\"odel's constructible universe $L$ is known as an inner model, that is, a definable class satisfying the same axioms (and containing the same ordinals). This gives a trivial proof that adding the axiom $V = L$ does not increase the consistency strength of the theory. In this paper, we shall look at a system of intuitionistic set theory known as $\mathrm{CZF}$, where $L$ fails to exhibit such nice properties. We will demonstrate that, here, the theory $\mathrm{CZF} + V = L$ is still equiconsistent with $\mathrm{CZF}$, but the proof will involve a much more complicated realisability model and a recursion-theoretic argument.
  \end{abstract}

\end{frontmatter}

The construction of G\"odel's constructible universe $L$ in intuitionistic logic was first studied by Lubarsky \cite{lubarsky93-intuitionistic-l} in intuitionistic Zermelo--Fraenkel set theory, $\mathrm{IZF}$, whereas Matthews and Rathjen \cite[section 5]{matthews-rathjen24-constructible-universe} verified recently that his constructions and fundamental results already work in the weak fragment known as intuitionistic Kripke--Platek set theory, $\mathrm{IKP}$. This is the first-order theory consisting of the axioms of extensionality, empty set, pairing, union, the axiom schemes of set induction, $\Delta_0$-comprehension and $\Delta_0$-collection, as well as the axiom of strong infinity\footnote{Some formulations of $\mathrm{IKP}$ would omit the infinity axiom. In Matthews \& Rathjen \cite{matthews-rathjen24-constructible-universe}, it is shown that when $L$ is constructed alternatively through Barwise's fundamental operations, the axiom of infinity is not necessary, and the resulting class coincides with Lubarsky's conventional construction through definable subsets. In this paper, we will be working with extensions of $\mathrm{IKP}$ that always admit the axiom of infinity, thus we shall stick to the usual definition instead.}.

Since intuitionistic set theories do not prove the classification of ordinals into zero, successors and limits, i.e.\ that every ordinal is either $\varnothing$, a successor or a non-empty set closed under successors, it is usually more appropriate to form recursive definitions with a uniform clause for all ordinals, i.e.
\[L_\alpha = \bigcup_{\beta \in \alpha} \dd\mleft(L_\beta\mright) \quad \text{for all $\alpha \in \Ord$},\]
where $\dd\mleft(X\mright)$ is the collection of definable sets in the first-order structure $\tuple{X; \in}$, and we then take $L = \bigcup_{\alpha \in \Ord} L_\alpha$. Lubarsky showed that, despite requiring a more complicated proof in the intuitionistic setting (due to the fact that the classical statement $\forall \alpha \in \Ord \ \alpha = L_\alpha \cap \Ord$ is no longer a theorem here), we still have
\begin{theorem}[$\mathrm{IKP}$, \cite{lubarsky93-intuitionistic-l,matthews-rathjen24-constructible-universe}]
  $L \vDash V = L$.
\end{theorem}

As an immediate consequence, let $T$ be any classical or intuitionistic extension of $\mathrm{IKP}$, then as long as we can show $T \vdash L \vDash T$, we know that
\[T \equiv_{\mathrm{Con}} T + V = L.\]
i.e.\ the two theories are equiconsistent. This is the case for well-known classical theories $\mathrm{KP}$, $\mathrm{ZF}$ and $\mathrm{ZFC}$; as we know from Lubarsky \cite{lubarsky93-intuitionistic-l} and Matthews \& Rathjen \cite{matthews-rathjen24-constructible-universe}, this is also the case for intuitionistic theories $\mathrm{IKP}$ and $\mathrm{IZF}$.

However, there are also known theories $T$ that do not satisfy $T \vdash L \vDash T$ and violate the equiconsistency claim above. For example, let $\mathrm{KP}\mleft(\mathcal{P}\mright)$ denote classical Kripke--Platek set theory with an additional powerset operation $\mathcal{P}$ defined as a new primitive function symbol, which is allowed to appear in comprehension and collection schemes. Rathjen \cite{rathjen20-power-kp-choice} then proved that
\begin{theorem}[{\cite[Theorem 4.6]{rathjen20-power-kp-choice}}]
  $\mathrm{KP}\mleft(\mathcal{P}\mright) + V = L$ is much stronger than $\mathrm{KP}\mleft(\mathcal{P}\mright)$.
\end{theorem}

On the other hand, important intuitionistic extensions of $\mathrm{IKP}$ include the constructive Zermelo--Fraenkel set theory, $\mathrm{CZF}$, which adds the axioms of unrestricted strong collection and subset collection (see \cite{aczel-rathjen10-cst-book} for details), as well as its further extension $\mathrm{CZF}\mleft(\mathcal{P}\mright) = \mathrm{CZF} + \mathrm{Powerset}$. Recently, it is shown in \cite{matthews-rathjen24-constructible-universe} that
\begin{theorem}[{\cite[Theorem 7.12]{matthews-rathjen24-constructible-universe}}]
  $\mathrm{CZF}\mleft(\mathcal{P}\mright) \nvdash L \vDash \mathrm{Exp}$, where $\mathrm{Exp}$ is the axiom of exponentiation stating
  \[\forall x, y \ \exists z \ \forall f \left(f \in z \leftrightarrow \text{$f$ is a function with domain $x$ and codomain $y$}\right).\]
\end{theorem}
$\mathrm{Exp}$ is a weakening of the powerset axiom and a theorem of $\mathrm{CZF}$. Consequently, for any fragment $T \subseteq \mathrm{CZF}\mleft(\mathcal{P}\mright)$ that proves $\mathrm{Exp}$, we do not have $T \vdash L \vDash T$. This leaves open the question whether we still have $T \equiv_{\mathrm{Con}} T + V = L$ for such theories, most notably when $T = \mathrm{IKP}\mleft(\mathcal{P}\mright)$, $\mathrm{CZF}$ and $\mathrm{CZF}\mleft(\mathcal{P}\mright)$.

In this paper, we shall resolve the case for $\mathrm{CZF}$ and prove that, still,
\[\mathrm{CZF} \equiv_{\mathrm{Con}} \mathrm{CZF} + V = L.\]
To do this, we will use a realisability model of $\mathrm{CZF}$ together with the axiom of subcountability,
\[\forall x \ \exists u, f \left(u \subseteq \omega \land \text{$f$ is a surjection from $u$ onto $x$}\right).\]
We then tweak this realisability model to include an $\omega$-sequence of \emph{strongly incomparable ordinals} in\footnote{In this paper, we treat the theory $\mathrm{CZF}$ which does not admit the axiom of powerset. Thus, throughout this paper, we shall use the notation $\mathcal{P}\mleft(x\mright)$ to denote a $\Delta_0$-definable class consisting of all subsets of $x$, which may not itself be a set.} $\mathcal{P}\mleft(\omega\mright)$ inside $L$, that is, a distinguished object $f_0 : \omega \rightarrow \mathcal{P}\mleft(\omega\mright) \cap \Ord$ in $L$ satisfying
\[\forall i, j \in \omega \left(f_0\mleft(i\mright) \subseteq f_0\mleft(j\mright) \rightarrow i = j\right).\]
Finally, we will show that this suffices for implementing a mechanism to encode any subcountable set using an ordinal in $L$ and thus conclude that the model satisfies $\mathrm{CZF} + V = L$.

\section{The realisability model}
\label{sec:real-model}

The model we consider in this paper will be a variant of the constructions in Rathjen \cite{rathjen14-czf-lpo}, that is, we make use of the well-known equiconsistency between $\mathrm{CZF}$ and the classical second-order arithmetic theory $\mathrm{BI}$, which consists of $\mathrm{ACA}_0$ (i.e.\ Peano arithmetic with second-order comprehension for arithmetic formulae and $\Sigma^0_1$-induction) and the axiom scheme of \emph{bar induction},
\[\forall X \ \mathrm{TI}\mleft(\prec, x \in X\mright) \rightarrow \mathrm{TI}\mleft(\prec, \varphi\mleft(x\mright)\mright) \quad \text{for any $\varphi\mleft(x\mright)$},\]
where $\prec$ is any arithmetic formula defining a binary relation on first-order objects and $\mathrm{TI}\mleft(\prec, \varphi\mleft(x\mright)\mright) = \forall x \left(\forall y \prec x \ \varphi\mleft(y/x\mright) \rightarrow \varphi\mleft(x\mright)\right) \rightarrow \forall x \ \varphi\mleft(x\mright)$. We will formulate a definable type structure in $\mathrm{BI}$ and realise $\mathrm{CZF} + \text{Subcountability}$ in it.

To begin, we equip $\mathbb{N}$ with a partial combinatory algebra (abbreviated henceforth as pca) using usual Turing computability, i.e.\ we write\footnote{As computable function application will be used pervasively here, we reserve the juxtaposition notation for convenience. To avoid ambiguity, multiplication on $\mathbb{N}$ will instead always be written with a dot, e.g.\ $p \cdot q$. Additionally, we will not explicitly use any natural number greater than $9$ in this paper, thus the notation $f11$ should be read as $\mleft(f1\mright)1$, instead of $f$ being applied to the number $11$.} $pq = r$ to mean that the Turing machine with index $p$ terminates on input $q$ and outputs the result $r$. We use $pq {\downarrow}$ to denote $\exists r \ pq = r$.

We define the following basic combinators in the pca:
\begin{itemize}
  \item $\mathbf{k}ab = a$;
  \item $\mathbf{s}abc = ac\mleft(bc\mright)$;
  \item $\mathbf{s}_N a = a + 1$;
  \item $\mathbf{p}_N a = \mathrm{max}\mleft\{a - 1, 0\mright\}$;
  \item $\mathbf{d}abc_1c_2 = \left\{\begin{aligned}
             & a &  & \text{if $c_1 = c_2$}, \\
             & b &  & \text{otherwise};
          \end{aligned}\right.$
  \item $\mathbf{p}ab = \left\{\begin{aligned}
             & \mathrm{max}\mleft\{a, b\mright\} \cdot \left(\mathrm{max}\mleft\{a, b\mright\} + 1\right) - a + b &  & \text{if $\mathrm{max}\mleft\{a, b\mright\}$ is even}, \\
             & \mathrm{max}\mleft\{a, b\mright\} \cdot \left(\mathrm{max}\mleft\{a, b\mright\} + 1\right) + a - b &  & \text{otherwise};
          \end{aligned}\right.$
  \item $\mathbf{p}_0\mleft(\mathbf{p}ab\mright) = a$;
  \item $\mathbf{p}_1\mleft(\mathbf{p}ab\mright) = b$.
\end{itemize}
Here, a specific definition for the pairing function $\mathbf{p}$ is needed for the subsequent construction. We first check that
\begin{lemma}[$\mathrm{PA}$]
  For any $n \in \mathbb{N}$, $n > 0$, the map $\tuple{a, b} \mapsto \mathbf{p}ab$ is a bijection between $\left\{0, \ldots, n - 1\right\}^2$ and $\left\{0, \ldots, n^2 - 1\right\}$.
\end{lemma}

\begin{proof}
  Just observe that for any $a, b \in \mathbb{N}$,
  \[\mathrm{max}\mleft\{a, b\mright\}^2 \leq \mathbf{p}ab < \left(\mathrm{max}\mleft\{a, b\mright\} + 1\right)^2,\]
  so for each $n \in \mathbb{N}$, $n > 0$, $\mathbf{p}$ bijects precisely the set of $\left(2 \cdot n - 1\right)$ tuples
  \[\left\{\tuple{n - 1, i} : 0 \leq i \leq n - 1\right\} \cup \left\{\tuple{i, n - 1} : 0 \leq i < n - 1\right\}\]
  onto $\left\{\left(n - 1\right)^2, \ldots, n^2 - 1\right\}$.
\end{proof}

It follows that $\mathbf{p}$ defines a recursive bijection $\mathbb{N}^2 \rightarrow \mathbb{N}$, which automatically ensures that the recursive inverses $\mathbf{p}_0, \mathbf{p}_1$ exist. We additionally observe that $\mathbf{p}$ satisfies the following property, which will become relevant for a recursion-theoretic argument in \autoref{sec:dist-ord} below:

\begin{lemma}
  \label{lem:pairing-property}
  For any $i, j \in \mathbb{N}$ where $i \neq j$, there exist arbitrarily large $n \in \mathbb{N}$ such that $\mathbf{p}in < \mathbf{p}jn$.
\end{lemma}

\begin{proof}
  From the definition of $\mathbf{p}$ it is easy to check that if $i < j$, then for any even $n > \mathrm{max}\mleft\{i, j\mright\}$ we will have $\mathbf{p}in < \mathbf{p}jn$; otherwise, $i > j$ and for any odd $n > \mathrm{max}\mleft\{i, j\mright\}$ we will have $\mathbf{p}in < \mathbf{p}jn$.
\end{proof}

We now define an instance of Aczel's Martin-L\"of type theory with a single $W$-type \cite{aczel78-type-theoretic-cst} (a theory now known as $\mathrm{ML}_1\mathrm{V}$), together with one arbitrary distinguished type, in $\mathrm{BI}$ and a realisability interpretation on top. This is analogous to the constructions in \cite[section 5--6]{rathjen14-czf-lpo} and is possible thanks to the availability of inductive definitions in $\mathrm{BI}$ (c.f.\ \cite[Lemma 3.2]{rathjen14-czf-lpo}). We fix an arbitrary set $X \subseteq \mathbb{N}$ and denote type names $\overline{n}, \overline{\mathbb{N}}, \overline{X}$ and type constructors $\sigma, \pi$ as elements of the pca:
\begin{alignat*}{2}
  \overline{n}          & = \mathbf{p}0n \quad \mathrlap{\text{for each $n \in \mathbb{N}$},}                                                                         \\
  \overline{\mathbb{N}} & = \mathbf{p}10,                                                     & \qquad \sigma & = \lambda nm. \mathbf{p}2\mleft(\mathbf{p}nm\mright), \\
  \overline{X}          & = \mathbf{p}11,                                                     & \pi           & = \lambda nm. \mathbf{p}3\mleft(\mathbf{p}nm\mright).
\end{alignat*}

We then construct $\mathbf{U}$ as the universe of (small) types, consisting of the names above and closed under dependent sums and products, and $\mathbf{V}$ as the type of well-founded trees defined inductively over $\mathbf{U}$. Formally, this involves building $\Pi^1_1$-classes $\mathbf{U}, \mathbf{V} \subseteq \mathbb{N}$ and relations ${\din}, {\dnin} \subseteq \mathbb{N} \times \mathbf{U}$ simultaneously in $\mathrm{BI}$ through the following inductive clauses:

\begin{enumerate}
  \item[$0$.] For any $n \in \mathbb{N}$, $\overline{n} \in \mathbf{U}$.

  \item[$0'$.] For any $k, n \in \mathbb{N}$, if $k < n$ then $k \din \overline{n}$; if $k \geq n$ then $k \dnin \overline{n}$.

  \item[$1$.] $\overline{\mathbb{N}}, \overline{X} \in \mathbf{U}$.

  \item[$1'$.] For any $k \in \mathbb{N}$, $k \din \overline{\mathbb{N}}$; if $k \in X$ then $k \din \overline{X}$; if $k \not\in X$ then $k \dnin \overline{X}$.

  \item[$2,3$.] For any $u \in \mathbf{U}$, $e \in \mathbb{N}$, if $\forall k \left(k \dnin u \lor \left(ek {\downarrow} \land ek \in \mathbf{U}\right)\right)$ then $\sigma ue, \pi ue \in \mathbf{U}$.

  \item[$2'$.] Assume $\sigma ue \in \mathbf{U}$, for any $k, \ell \in \mathbb{N}$, if $k \din u$ and $\ell \din ek$ then $\mathbf{p}k \ell \din \sigma ue$; if $k \dnin u$ or $ \ell \dnin ek$ then $\mathbf{p}k \ell \dnin \sigma ue$.

  \item[$3'$.] Assume $\pi ue \in \mathbf{U}$, for $d \in \mathbb{N}$, if $\forall k \left(k \dnin u \lor \left(dk {\downarrow} \land dk \din ek\right)\right)$ then $d \din \pi ue$; if $\exists k \left(k \din u \land \left(dk {\downarrow} \rightarrow dk \dnin ek\right)\right)$ then $d \dnin \pi ue$.

  \item[$4$.] For any $u \in \mathbf{U}$, $e \in \mathbb{N}$, if $\forall k \left(k \dnin u \lor \left(ek {\downarrow} \land ek \in \mathbf{V}\right)\right)$ then $\mathbf{p}ue \in \mathbf{V}$.
\end{enumerate}

Clearly, the rules stipulate that every $\overline{n} \in \mathbf{U}$ denotes a finite type with $n$ elements; $\overline{\mathbb{N}} \in \mathbf{U}$ denotes the type of natural numbers. The type constructors $\sigma$ and $\pi$ builds dependent sums and products respectively. Notably, since the elements of $\pi$-types are required specifically to be computable functions, by manipulating the computational complexity of the set $X$ we can alter the content of derived $\pi$-types.

For example, let $h : \mathbb{N} \rightarrow \mathbb{N}$ be any function and $X = \left\{t \in \mathbb{N} : t \sqsubseteq h\right\}$ be the set of natural number codes for the finite initial segments of $h$ such that $\mathbf{l}_s : X \rightarrow \mathbb{N}$ is a recursive function mapping each initial segment to its length (c.f.\ the beginning of \autoref{sec:dist-ord} for a more detailed set-up), then the type
\[\pi \overline{\mathbb{N}} \mleft(\lambda n. \sigma \overline{X}\mleft(\lambda t. \mathbf{d} \overline{1} \overline{0} n \mleft(\mathbf{l}_s t\mright)\mright)\mright)\]
represents precisely the computable functions that enumerate initial segments of $h$, and thus is non-empty if and only if $h$ is computable. This shall be taken advantage of in \autoref{sec:dist-ord} below.

The realisability model on this type system shall interpret the trees in $\mathbf{V}$ as sets, and we first need to define the collection of realisers for extensional equality as a family of types, $\left(\alpha \deq \beta\right) \in \mathbf{U}$ for any $\alpha, \beta \in \mathbf{V}$, through the proposition-as-type interpretation. Formally, by the recursion theorem for Turing computations, we can define a binary function $\alpha, \beta \mapsto \alpha \deq \beta$ using a term in the pca to satisfy the mutual recursive definitions
\begin{align*}
  \alpha \deq \beta \quad  & = \quad \sigma \mleft(\alpha \dseq \beta\mright) \mleft(\lambda x. \beta \dseq \alpha\mright),                                                                                              \\
  \alpha \dseq \beta \quad & = \quad \pi \mleft(\mathbf{p}_0 \alpha\mright) \mleft(\lambda x. \sigma \mleft(\mathbf{p}_0 \beta\mright) \mleft(\lambda y. \mathbf{p}_1 \alpha x \deq \mathbf{p}_1 \beta y\mright)\mright)
\end{align*}
for any $\alpha, \beta \in \mathbf{V}$. The totality of this function on $\mathbf{V}^2$ follows from induction on $\mathbf{V}$.

We can now define the following realisability conditions, where the first-order parameters in the realised formula range over objects in $\mathbf{V}$:
\begin{align*}
  e \Vdash \alpha = \beta \quad                                 & \coloneqq \quad e \din \left(\alpha \deq \beta\right),                                                                                                       \\
  e \Vdash \alpha \in \beta \quad                               & \coloneqq \quad \mathbf{p}_0 e \din \mathbf{p}_0 \beta \land \mathbf{p}_1 e \Vdash \alpha = \mathbf{p}_1 \beta \mleft(\mathbf{p}_0 e\mright),                \\
  e \Vdash \neg \varphi \quad                                   & \coloneqq \quad \forall d \ \neg d \Vdash \varphi,                                                                                                           \\
  e \Vdash \varphi \land \psi \quad                             & \coloneqq \quad \mathbf{p}_0 e \Vdash \varphi \land \mathbf{p}_1 e \Vdash \psi,                                                                              \\
  e \Vdash \varphi \lor \psi \quad                              & \coloneqq \quad \left(\mathbf{p}_0 e = 0 \land \mathbf{p}_1 e \Vdash \varphi\right) \lor \left(\mathbf{p}_0 e = 1 \land \mathbf{p}_1 e \Vdash \psi\right),   \\
  e \Vdash \varphi \rightarrow \psi \quad                       & \coloneqq \quad \forall d \ \left(d \Vdash \varphi \rightarrow ed {\downarrow} \land ed \Vdash \psi\right),                                                  \\
  e \Vdash \forall x \in \alpha \ \varphi\mleft(x\mright) \quad & \coloneqq \quad \forall i \din \mathbf{p}_0 \alpha \left(ei {\downarrow} \land ei \Vdash \varphi\mleft(\mathbf{p}_1 \alpha i\mright)\right),                 \\
  e \Vdash \exists x \in \alpha \ \varphi\mleft(x\mright) \quad & \coloneqq \quad \mathbf{p}_0 e \din \mathbf{p}_0 \alpha \land \mathbf{p}_1 e \Vdash \varphi\mleft(\mathbf{p}_1 \alpha \mleft(\mathbf{p}_0 e\mright)\mright), \\
  e \Vdash \forall x \ \varphi\mleft(x\mright) \quad            & \coloneqq \quad \forall \alpha \in \mathbf{V} \left(e\alpha {\downarrow} \land e\alpha \Vdash \varphi\mleft(\alpha\mright)\right),                           \\
  e \Vdash \exists x \ \varphi\mleft(x\mright) \quad            & \coloneqq \quad \mathbf{p}_0 e \in \mathbf{V} \land \mathbf{p}_1 e \Vdash \varphi\mleft(\mathbf{p}_0 e\mright).
\end{align*}

It is routine to verify in $\mathrm{BI}$ that the axioms of $\mathrm{CZF}$ are all realised in this interpretation, for example following the detailed working in \cite{aczel78-type-theoretic-cst} or \cite{rathjen06-formulae-as-classes-interpretation}. We will mention below a few common constructions in the literature that will become convenient for our later use:

Firstly, by the recursion theorem, we can easily construct some $\iota, \delta \in \mathbb{N}$ such that
\[\delta n = \mathbf{p} n \iota, \qquad \iota = \mathbf{p} \delta \delta.\]
By induction on $\mathbf{V}$, $\iota \Vdash \alpha = \alpha$ for any $\alpha \in \mathbf{V}$.

Secondly, the canonical unordered pair is given by $\left\{\alpha, \beta\right\}^V = \mathbf{p}\overline{2}\mleft(\lambda x. \mathbf{d}\alpha\beta x0\mright)$. The canonical ordered pair can then simply be $\tuple{\alpha, \beta}^V = \left\{\left\{\alpha, \alpha\right\}^V, \left\{\alpha, \beta\right\}^V\right\}^V$. These representations are all computable in the pca.

Also, for $n \in \mathbb{N}$, we can recursively construct $n^V = \mathbf{p}\overline{n}\mleft(\lambda x. x^V\mright)$, then we also have $\omega^V = \mathbf{p}\overline{\mathbb{N}}\mleft(\lambda x. x^V\mright)$. The object $\omega^V \in \mathbf{V}$ will be what the realisability model thinks is the ordinal $\omega$.

We additionally verify that the axiom of subcountability is realised:

\begin{proposition}[Subcountability]
  There exists $e \in \mathbb{N}$ such that
  \[e \Vdash \forall x \ \exists u, f \left(u \subseteq \omega \land \text{$f$ is a surjection from $u$ onto $x$}\right).\]
\end{proposition}

\begin{proof}
  For any $\alpha \in \mathbf{V}$ that is assigned to $x$, we can compute $u$ as
  \[\mathbf{p} \mleft(\mathbf{p}_0 \alpha\mright) \mleft(\lambda x. x^V\mright),\]
  and $f$ as
  \[\mathbf{p} \mleft(\mathbf{p}_0 \alpha\mright) \mleft(\lambda x. \tuple{x^V, \mathbf{p}_1 \alpha x}^V\mright).\]
  It is easy to check that ``$u \subseteq \omega \land \text{$f$ is a surjection from $u$ onto $x$}$'' is realised by  an appropriate object.
\end{proof}

\section{The distinguished objects}
\label{sec:dist-ord}

We will now select a specific set $X$ through a recursion-theoretic construction in $\mathrm{BI}$ and use $\overline{X}$ to build some $\alpha_0 \in \mathbf{V}$ that the realisability model thinks is an intuitionistic ordinal with very specific behaviours. We fix an arbitrary, recursive way to code any finite $\mathbb{N}$-sequence as a single natural number, such that the following notions are definable as applicative terms in the pca:
\begin{itemize}
  \item $\mathbf{l}_s t$ is the length of the sequence,
  \item $\mathbf{t}_s tn$ is the truncated sequence of first $n$ components in $t$ for any $0 \leq n \leq \mathbf{l}_s t$,
\end{itemize}
whenever $t$ is a valid code for some $\mathbb{N}$-sequence.

We consider a priority argument where for each tuple of $i, j, f \in \mathbb{N}$ where $i \neq j$, we have the following requirement on a finite sequence $t$: there exists some $n \in \mathbb{N}$ such that $i, j < n$, $\mathbf{p}in, \mathbf{p}jn \leq \mathbf{l}_s t$, and
\[f\mleft(\mathbf{t}_s t \mleft(\mathbf{p}in\mright)\mright) {\downarrow} \rightarrow f\mleft(\mathbf{t}_s t \mleft(\mathbf{p}in\mright)\mright) \neq \mathbf{t}_s t \mleft(\mathbf{p}jn\mright),\]
i.e.\ that the Turing machine with index $f$ does not correctly compute the first $\mathbf{p}jn$ components of the sequence $t$ given its first $\mathbf{p}in$ components, for some $n \in \mathbb{N}$ where $t$ has enough components. Observe we can show that

\begin{proposition}[$\mathrm{ACA}_0$]
  \label{prop:non-comp-path}
  There is some $h : \mathbb{N} \rightarrow \mathbb{N}$ such that for any $i, j, f \in \mathbb{N}$ where $i \neq j$, the corresponding requirement is satisfied by some finite initial segment of $h$.
\end{proposition}

\begin{proof}
  Finite sequences $t$ satisfying each requirement forms an upwards closed set, and \autoref{lem:pairing-property} ensures that this set will be dense: for any $i, j \in \mathbb{N}$ where $i \neq j$  and sequence $t$, take some $n > i, j, \mathbf{l}_s t$ such that $\mathbf{p}in < \mathbf{p}jn$ (so we also have $\mathbf{p}jn > \mathbf{l}_s t$) and extend $t$ if needed to any $t' \sqsupseteq t$ such that $\mathbf{p}in \leq \mathbf{l}_s t' < \mathbf{p}jn$. Then for any further extension $t'' \sqsupseteq t'$, the value $f\mleft(\mathbf{t}_s t'' \mleft(\mathbf{p}in\mright)\mright) = f\mleft(\mathbf{t}_s t' \mleft(\mathbf{p}in\mright)\mright)$ (or the non-termination thereof) is already determined. Clearly by choosing a suitable extension $t''$ of length $\mathbf{p}jn$, we can make sure that $f$ does not compute $\mathbf{t}_s t'' \mleft(\mathbf{p}jn\mright) = t''$ correctly.

  Therefore, we can enumerate all desired requirements and, by arithmetic comprehension, find a sequence of finite sequences $\left\{t_i\right\}_{i \in \mathbb{N}}$ such that each $t_i$ satisfies the first $i$ requirements and $t_{i + 1} \sqsupseteq t_i$ is an extension. Combining these finite sequences immediately give a path $h : \mathbb{N} \rightarrow \mathbb{N}$ such that every requirement above is satisfied by some finite initial segment of $h$.
\end{proof}

We now take our distinguished type $\overline{X}$ in the realisability model to be given by
\[X = \left\{t \in \mathbb{N} : t \sqsubseteq h\right\},\]
i.e.\ the (coded) finite initial segments of $h$. Consider $\alpha_0 = \mathbf{p}\overline{X}\mleft(\lambda t. \left(\mathbf{l}_s t\right)^V\mright) \in \mathbf{V}$ in the realisability model. We first see that it has the following basic properties:

\begin{lemma}
  The following are realised:
  \begin{enumerate}[label=(\roman*)]
    \item $\alpha_0 \in \Ord$;
    \item $\alpha_0 \subseteq \omega$.
  \end{enumerate}
\end{lemma}

\begin{proof}
  First check that
  \[\lambda t. \mathbf{p}\mleft(\mathbf{l}_s t\mright)\iota \Vdash \forall \alpha \in \alpha_0 \ \alpha \in \omega^V,\]
  thus (ii) is realised.

  Now, for (i), observe that for any $t \in X$ and $0 \leq n \leq \mathbf{l}_s t$, $\mathbf{p}\mleft(\mathbf{t}_s tn\mright) \iota \Vdash n^V \in \alpha_0$. Therefore,
  \[\lambda tx. \mathbf{p}\mleft(\mathbf{t}_s tx\mright)\iota \Vdash \forall \alpha \in \alpha_0 \ \forall \beta \in \alpha \ \beta \in \alpha_0,\]
  realising that $\alpha_0$ is transitive. Now it follows immediately from $\alpha_0 \subseteq \omega$ within $\mathrm{CZF}$ that any element in $\alpha_0$ is transitive, thus $\alpha_0 \in \Ord$.
\end{proof}

Next, we use the fact that our pairing function $\mathbf{p}$ is recursive; this means we can write down the same definition inside the realisability model (or even inside $L$) to define a set-theoretic function from $\omega^2$ to $\omega$ that performs exactly the same computations. More explicitly, this set-theoretic pairing function could bear the name
\[\overline{p} = \mathbf{p}\overline{\mathbb{N}} \mleft(\lambda x. \tuple{\tuple{\left(\mathbf{p}_0 x\right)^V, \left(\mathbf{p}_1 x\right)^V}^V, x^V}^V\mright) \in \mathbf{V}.\]

We use $\overline{p}$ inside the realisability model to extract the complexity coded into $\alpha_0$ in \autoref{prop:non-comp-path}. We define the set-theoretic function $f_0 : \omega \rightarrow \mathcal{P}\mleft(\omega\mright)$ by
\[f_0\mleft(i\mright) = \left\{n \in \omega : \forall k \in n \ \overline{p}\mleft(i, k\mright) \in \alpha_0\right\}.\]
It is easy to check within $\mathrm{CZF}$ that this definition guarantees each $f_0\mleft(i\mright) \in \Ord$ as well. We also show that

\begin{proposition}
  \label{prop:f-subset-eq}
  For any $i, j \in \mathbb{N}$, if some $d \Vdash f_0\mleft(i^V\mright) \subseteq f_0\mleft(j^V\mright)$, then $i = j$. Namely, we trivially have
  \[\lambda ijx. \iota \Vdash \forall i, j \in \omega \left(f_0\mleft(i\mright) \subseteq f_0\mleft(j\mright) \rightarrow i = j\right).\]
\end{proposition}

\begin{proof}
  Assume otherwise, i.e.\ $i \neq j$. Within $\mathrm{CZF}$, we know that $f_0\mleft(i^V\mright) \subseteq f_0\mleft(j^V\mright)$, by definition, is equivalent to
  \[\forall n \in \omega \left(\forall k \in n \ \overline{p}\mleft(i^V, k\mright) \in \alpha_0 \rightarrow \forall k \in n \ \overline{p}\mleft(j^V, k\mright) \in \alpha_0\right),\]
  so suppose that this is realised by some $d \in \mathbb{N}$. Now, we take the formula $\overline{p}\mleft(i^V, k^V\mright) \in \alpha_0$ as a shorthand for
  \[\exists x \in \overline{p} \ \exists n \in \alpha_0 \ x = \tuple{\tuple{i^V, k^V}, n},\]
  thus if some $e \Vdash \overline{p}\mleft(i^V, k^V\mright) \in \alpha_0$, then $\mathbf{p}_0 e = \mathbf{p}ik$, and $\mathbf{p}_0 \mleft(\mathbf{p}_1 e\mright)$ must be the unique finite sequence in $X$ of length precisely $\mathbf{p}ik$.

  Conversely, there must also be some function $q \in \mathbb{N}$ in the pca such that, suppose $t \in X$ is any finite sequence satisfying $\mathbf{l}_s t \geq \mathbf{p}ik$,
  \[\mathbf{p} \mleft(\mathbf{p}ik\mright) \mleft(\mathbf{p} \mleft(\mathbf{t}_s t \mleft(\mathbf{p}ik\mright)\mright) \mleft(qik\mright)\mright) \Vdash \overline{p}\mleft(i^V, k^V\mright) \in \alpha_0.\]
  (Here we will omit the cumbersome process of writing out the internal structure of realisers for the defining formula of set-theoretic ordered pairs, but it is routine to check that this can be done.)

  Consequently, for any $n > \mathrm{max}\mleft\{i, j\mright\}$, let $t \in X$ be any long enough finite sequence such that $\mathbf{p}in, \mathbf{p}jn \leq \mathbf{l}_s t$. Observe that our stipulation of the pairing function $\mathbf{p}$ also ensures $\forall k \leq n \ \mathbf{p}ik \leq \mathbf{p}in$. Therefore,
  \[\lambda k. \mathbf{p} \mleft(\mathbf{p}ik\mright) \mleft(\mathbf{p} \mleft(\mathbf{t}_s t \mleft(\mathbf{p}ik\mright)\mright) \mleft(qik\mright)\mright) \Vdash \forall k \in \left(\mathbf{s}_N n\right)^V \overline{p}\mleft(i^V, k\mright) \in \alpha_0.\]
  In other words, if we take $f = \lambda nt. d\mleft(\mathbf{s}_N n\mright) \mleft(\lambda k. \mathbf{p} \mleft(\mathbf{p}ik\mright) \mleft(\mathbf{p} \mleft(\mathbf{t}_s t \mleft(\mathbf{p}ik\mright)\mright) \mleft(qik\mright)\mright)\mright) n$, then we have $fnt \Vdash \overline{p}\mleft(j^V, n^V\mright) \in \alpha_0$ and $\mathbf{p}_0 \mleft(\mathbf{p}_1 \mleft(fnt\mright)\mright) = \mathbf{t}_s t \mleft(\mathbf{p}jn\mright)$.

  Observe that $\mathbf{l}_s \mleft(\mathbf{t}_s t \mleft(\mathbf{p}in\mright)\mright) = \mathbf{p}in$, i.e.\ $\mathbf{p}_1 \mleft(\mathbf{l}_s \mleft(\mathbf{t}_s t \mleft(\mathbf{p}in\mright)\mright)\mright) = n$. This means we can have a single-parameter function $g = \lambda s. \mathbf{p}_0\mleft(\mathbf{p}_1\mleft(f \mleft(\mathbf{p}_1 \mleft(\mathbf{l}_s s\mright)\mright) s\mright)\mright)$, such that
  \[g\mleft(\mathbf{t}_s t \mleft(\mathbf{p}in\mright)\mright) {\downarrow} \land g\mleft(\mathbf{t}_s t \mleft(\mathbf{p}in\mright)\mright) = \mathbf{t}_s t \mleft(\mathbf{p}jn\mright).\]

  However, this recursive $g$ is now universal for any appropriate $n$ and long enough $t$, whereas in our construction of $X$ we have required that there exist $n$ and $t$ where $g$ does not compute $\mathbf{t}_s t \mleft(\mathbf{p}jn\mright)$ correctly. Therefore, $i \neq j$ implies a contradiction and thus we must have $i = j$.
\end{proof}

\section{The coding mechanism}

We now work entirely within the realised theory of $\mathrm{CZF} + \text{Subcountability}$ with our distinguished ordinal $\alpha_0$ to prove that $V = L$. In the next two arguments, we follow set-theoretic conventions and write $\alpha^+ = \alpha \cup \left\{\alpha\right\}$ to denote the successor of an ordinal.

Firstly, we show that
\begin{lemma}[$\mathrm{CZF}$]
  \label{lem:subset-omega-l}
  $\mathcal{P}\mleft(\omega\mright) \cap \Ord \subseteq L$. Namely, for any ordinal $\alpha \in \mathcal{P}\mleft(\omega\mright)$, we have $\alpha = L_\alpha \cap \Ord$.
\end{lemma}

\begin{proof}
  Lubarsky proved in \cite{lubarsky93-intuitionistic-l} that precisely $L_\omega \cap \Ord = \omega$. Using the additional fact that for any $n \in \omega$, any $\alpha \in L_{n^+} \cap \Ord$ is a subset of $L_n \cap \Ord$, we can easily show by induction that $\forall n \in \omega \ L_n \cap \Ord = n$. It follows that
  \[L_\alpha \cap \Ord = \bigcup_{n \in \alpha} \dd\mleft(L_n\mright)  \cap \Ord = \bigcup_{n \in \alpha} L_{n^+} \cap \Ord = \bigcup_{n \in \alpha} n^+ = \alpha. \qedhere\]
\end{proof}

This means that $\alpha_0 \subseteq \omega$ defined in \autoref{sec:dist-ord} lies in $L$. Since $\overline{p}$ is recursive and $f_0$ is defined from $\alpha_0$ and $\overline{p}$, correspondingly $f_0 \in L$.

\begin{proposition}
  \label{prop:omega-power-in-l}
  In the realisability model, we have $\mathcal{P}\mleft(\omega\mright) \subseteq L$.
\end{proposition}

\begin{proof}
  For any $u \subseteq \omega$, we consider $\alpha = \bigcup_{n \in u} \dd\mleft(L_{f_0\mleft(n\mright)}\mright) \cap \Ord$. Observe that
  \[\alpha = \bigcup_{n \in u} L_{f_0\mleft(n\mright)^+} \cap \Ord = L_{\bigcup_{n \in u} f_0\mleft(n\mright)^+} \cap \Ord \in L.\]
  We take some large enough $\eta \in \Ord$ such that $f_0, \omega, \alpha \in L_\eta$, and let
  \[v = \left\{n \in \omega : f_0\mleft(n\mright) \in \alpha\right\} \in \dd\mleft(L_\eta\mright).\]
  Since each $f_0\mleft(n\mright) \in \mathcal{P}\mleft(\omega\mright)$, by \autoref{lem:subset-omega-l}, also each $f_0\mleft(n\mright) = L_{f_0\mleft(n\mright)} \cap \Ord \in \dd\mleft(L_{f_0\mleft(n\mright)}\mright) \cap \Ord$, thus $u \subseteq v$. On the other hand, for any $m \in v$, there exists $n \in u$ such that $f_0\mleft(m\mright) \in \dd\mleft(L_{f_0\mleft(n\mright)}\mright) \cap \Ord$, i.e.\ $f_0\mleft(m\mright) \subseteq L_{f_0\mleft(n\mright)} \cap \Ord = f_0\mleft(n\mright)$. This immediately implies $m = n$ by \autoref{prop:f-subset-eq} and thus $m \in u$. Therefore, $u = v \in L$.
\end{proof}

\begin{theorem}
  In the realisability model, we have $V = L$.
\end{theorem}

\begin{proof}
  For any set $s$ consider the transitive closure $\trcl\mleft(\left\{s\right\}\mright)$ and correspondingly a surjection $f : u \rightarrow \trcl\mleft(\left\{s\right\}\mright)$ where $u \subseteq \omega$, by the axiom of subcountability. Without loss of generality we can assume conveniently that $0 \in u$ and $f\mleft(0\mright) = s$. We define
  \[\sigma_s = \left\{n \in \omega : \exists i, j \in u \left(f\mleft(i\mright) \in f\mleft(j\mright) \land n = \overline{p}\mleft(i, j\mright)\right)\right\}.\]

  Given $\sigma_s$, we can recursively define some (possibly partial) function $g$ on domain $u$ as
  \[g\mleft(n\mright) = \left\{g\mleft(m\mright) : m \in u \land \overline{p}\mleft(m, n\mright) \in \sigma_s\right\}\]
  if $g\mleft(m\mright)$ exists for all such $m \in u$. The function $g$ is defined by a $\Sigma_1$-formula and parameters $u, \sigma_s$, where $u, \sigma_s \in L$ by \autoref{prop:omega-power-in-l}. Using $\Delta_0$-collection, we can easily prove by set induction on $x$ that, for any set $x$,
  \[\forall n \in u \left(f\mleft(n\mright) = x \rightarrow g\mleft(n\mright) = x \land L \vDash \text{$g\mleft(n\mright)$ exists}\right).\]
  Since $g$, as a $\Sigma_1$-definable function, must be upwards absolute between transitive classes, clearly $s = g\mleft(0\mright) \in L$.
\end{proof}

We have thus shown that $\mathrm{BI}$ realises $\mathrm{CZF} + V = L$. Using the fact that $\mathrm{BI} \equiv_{\mathrm{Con}} \mathrm{CZF}$ (c.f.\ \cite[Theorem 2.2]{rathjen14-czf-lpo}), we effectively know that $\mathrm{CZF}$ is equiconsistent with $\mathrm{CZF} + V = L$.

\section*{Acknowledgements}

The author thanks the support from the UK Engineering and Physical Sciences Research Council [EP/W523860/1] during the preparation of this paper.

\bibliographystyle{elsarticle-num}
\bibliography{\jobname}

\end{document}